\documentclass[11pt, a4paper, reqno]{amsart}
\author{Soichiro Fujii}

\date{\today}

\address{Research Institute for Mathematical Sciences, Kyoto University\\
 Kyoto 606-8502, Japan}

\title {Enriched categories and tropical mathematics}


\email{s.fujii.math@gmail.com}

\usepackage[centering, totalwidth = 380pt, totalheight = 600pt]{geometry}
\usepackage{mathrsfs}
\usepackage{amsfonts, amssymb}
\usepackage{mathtools}
\usepackage{comment}
\usepackage{thmtools}
\usepackage{txfonts}
\usepackage{lmodern}
\usepackage{multirow}
\usepackage{tikz}
\usetikzlibrary{arrows.meta,decorations.pathreplacing,decorations.markings,shapes,calc}
\usetikzlibrary{matrix,arrows,decorations.pathmorphing,backgrounds,decorations.markings,positioning}
\tikzset{2cell/.style={-implies,double,double equal sign distance,shorten 
>=2pt, shorten <=3pt}}
\tikzset{2cellshort/.style={-implies,double,double equal sign distance,shorten 
>=4pt, shorten <=5pt}}
\tikzset{2cellr/.style={implies-,double,double equal sign distance,shorten 
>=3pt, shorten <=2pt}}
\tikzset{3cell/.style={-implies,double,double distance=2.5pt,shorten >=2pt, 
shorten <=3pt}}
\tikzset{labelsize/.style={font=\scriptsize}}
\tikzset{string/.style={very thick}}
\tikzset{
  pto/.style={->,postaction={decorate},
    decoration={
        markings,
        mark=at position 0.5 with {\arrow{|}}}
  },
}
\newcommand{\tzdarrow}[2]{
\draw[2cell] (#1,#2)++(0,0.3) to ++(0,-0.6);}

\usepackage[colorlinks=true,pagebackref=true]{hyperref}
\declaretheoremstyle[
spaceabove=6pt, spacebelow=6pt,
headfont=\normalfont\itshape,
notefont=\mdseries, notebraces={(}{)},
bodyfont=\normalfont,
postheadspace=0.5em,
qed=$\square$
]{myproofstyle}

\declaretheorem[style=plain,numberwithin=section,name=Theorem]{theorem}

\declaretheorem[style=plain,sibling=theorem,name=Proposition]{proposition}
\declaretheorem[style=plain,sibling=theorem,name=Corollary]{corollary}

\declaretheorem[style=definition,qed=$\blacksquare$,sibling=theorem,name=Definition]{definition}
\declaretheorem[style=definition,qed=$\blacksquare$,sibling=theorem,name=Example]{example}
\makeatletter
\newcommand{\dotminus}{\mathbin{\text{\@dotminus}}}

\newcommand{\@dotminus}{%
  \ooalign{\hidewidth\raise1ex\hbox{.}\hidewidth\cr$\m@th-$\cr}%
}
\makeatother

\mathchardef\mhyphen="2D

\makeatletter
\newcommand{\pto}{}% just for safety
\newcommand{\pgets}{}% just for safety
\DeclareRobustCommand{\pto}{\mathrel{\mathpalette\p@to@gets\to}}
\DeclareRobustCommand{\pgets}{\mathrel{\mathpalette\p@to@gets\gets}}
\newcommand{\p@to@gets}[2]{%
  \ooalign{\hidewidth$\m@th#1\mapstochar\mkern5mu$\hidewidth\cr$\m@th#1\longrightarrow$\cr}%
}
\makeatother

\newcommand{\id}[1]{{\mathrm{id}_{#1}}}
\newcommand{\ob}[1]{{\mathrm{ob}(#1)}}

\newcommand{\op}{{\mathrm{op}}}

\newcommand{\rev}{{\mathrm{rev}}}
\newcommand{\revop}{{\mathrm{revop}}}

\newcommand{\cat}[1]{{\mathcal{#1}}}

\newcommand{\defemph}[1]{\textbf{#1}}

\newcommand{\zero}{\mathbf{0}}
\newcommand{\one}{I}
\newcommand{\tms}{\circ}
\newcommand{\pls}{\oplus}

\newcommand{\mto}{\mathbin{\searrow}}
\newcommand{\mot}{\mathbin{\swarrow}}

\newcommand{\Mat}[1]{{{#1}\mhyphen\mathbf{Mat}}}

\newcommand{\Rmax}{{\overline{\mathbb{R}}_{\mathrm{max},+}}}
\newcommand{\Rmin}{{\overline{\mathbb{R}}_{\mathrm{min},+}}}
\newcommand{\nonnegRmin}{{\overline{\mathbb{R}}^{\geq 0}_{\mathrm{min},+}}}

\newcommand{\nonnegRminmax}{{\overline{\mathbb{R}}^{\geq 0}_{\mathrm{min,max}}}}
\newcommand{\TV}{{\mathbf{2}}}

\newcommand{\End}{\mathrm{End}}

\newcommand{\Pow}{\mathcal{P}}

\newcommand{\Isb}{\mathrm{Isb}}
\newcommand{\Isbpos}{\mathcal{I}sb}
\newcommand{\Fix}{\mathrm{Fix}}
\newcommand{\Fixpos}{\mathcal{F}ix}

\newcommand{\tr}{\top}
\newcommand{\ind}[1]{\chi_{#1}}

\newcommand{\act}{\ast}

\newcommand{\Tr}{\mathrm{Tr}}

\newcommand{\pairing}[2]{{\langle #1,#2\rangle}}

\begin{document}
\begin{abstract}
This is a survey paper on the connection of enriched category theory over a quantale and tropical mathematics.
Quantales or complete idempotent semirings, as well as matrices with coefficients in them, are fundamental objects in both fields.
We first explain standard category-theoretic constructions on matrices, namely composition, right extension, right lifting and the Isbell hull.
Along the way, we review known reformulations (due to Elliott and Willerton) of tropical polytopes, directed tight spans and the Legendre--Fenchel transform by means of these constructions, illustrating their ubiquity in tropical mathematics and related fields.
We then consider complete semimodules over a quantale $\cat{Q}$, a tropical analogue of vector spaces over a field, and mention Stubbe's result identifying them with skeletal and complete $\cat{Q}$-categories.
With the aim to bridge a gap between enriched category theory and tropical mathematics, we assume no knowledge in either field.
\end{abstract}
\maketitle
\section{Introduction}
\emph{Tropical mathematics} (also called idempotent mathematics) is mathematics over an idempotent semiring, instead of over the fields of real or complex numbers as in many branches of ordinary mathematics.
A prototypical example of idempotent semirings is the so-called \emph{min-plus semiring} $([-\infty,\infty],\infty, 0, \min, +)$, defined as the set $[-\infty,\infty]$ of extended real numbers equipped with the minimum operation as (idempotent) addition and the ordinary addition $+$ as multiplication.
Such structures naturally arise in many branches of mathematical sciences, including operations research \cite{CuninghameGreen_minmax_appl}, formal language theory \cite{Simon_tropical,Pin_tropical}, mathematical physics \cite{LMS_idem_fa} and algebraic geometry \cite{Viro_dequantization}.

\emph{Enriched categories}, on the other hand, are a generalisation of categories \cite{MacLane_CWM} defined relative to a \emph{base}.
We can choose arbitrary \emph{monoidal categories} as the base, of which idempotent semirings are a special case. The theory works best when the base enjoys certain {completeness} properties, and in the case of idempotent semiring, imposing these additional completeness assumptions results in the notion of (unital) \emph{quantale} \cite{Mulvey_and}.
Preordered sets, metric spaces and ultrametric spaces are all instances of enriched categories over a quantale \cite{Lawvere_metric}.

In this paper we take a look at a branch of tropical mathematics, namely \emph{linear algebra over a quantale} (cf.~\cite{CGQ_duality}), from the enriched categorical perspective.
We believe that enriched category theory can provide an abstract framework to tropical mathematics, which would help us to recognise various easy facts that hold for purely formal reasons, and focus on genuine problems.
On the other hand, tropical mathematics can offer novel interpretation to category theoretic concepts, which might eventually lead to discovery of useful categorical notions of tropical origin, along the lines of Lawvere's notion of Cauchy completeness generalised from metric spaces to enriched categories \cite{Lawvere_metric}.
The purpose of this paper is to lay the foundations for a further fruitful interplay between enriched categories and tropical mathematics, by collecting basic notions and surveying known results which suggest various links between the two fields.
Accordingly, we have tried to make this paper mostly self-contained and accessible to both category theorists and tropical mathematicians.

\medskip

The outline of this paper is as follows.
After reviewing the notion of quantale (Section~\ref{sec:quantale}), we move on to study matrices over a quantale $\cat{Q}$ (called $\cat{Q}$-matrices) in Section~\ref{sec:matrices}.
Matrices over a quantale admit not only composition, but also the operations of \emph{right extension} and \emph{right lifting}, which may be computed by the formulas dual to the one for composition. Right extensions and right liftings of $\cat{Q}$-matrices enable us to define the \emph{Isbell hull} of a $\cat{Q}$-matrix (Section~\ref{sec:Isbell_hull}), which is known to unify such constructions as directed tight spans of metric spaces \cite{Willerton_tight_span}, lower semicontinuous convex functions on $\mathbb{R}^n$ \cite{Willerton_Legendre} and tropical polytopes \cite{Elliott_thesis}; we also provide brief expositions of these results.

In Section \ref{sec:semimodule} we study complete semimodules over a quantale (cf.~\cite{CGQ_duality,Belohlavek_Konecny}), a tropical analogue of vector spaces over a field. 
It turns out that the Isbell hull of a $\cat{Q}$-matrix naturally admits the structure of a complete semimodule over $\cat{Q}$, and conversely, any complete semimodule over $\cat{Q}$ can be realised as the Isbell hull of some $\cat{Q}$-matrix. 
In Section \ref{sec:Q-cat} we introduce categories enriched over a quantale $\cat{Q}$ (called $\cat{Q}$-categories). Any complete semimodule over $\cat{Q}$ naturally induces a $\cat{Q}$-category, and those $\cat{Q}$-categories arising from a complete semimodule over $\cat{Q}$ are characterised by the intrinsic properties of \emph{skeletality} and \emph{completeness} \cite{Stubbe_tensor_cotensor}. 

\subsection*{Acknowledgements}
We are grateful to Jonathan Elliott, Masahito Hasegawa and Simon Willerton 
for their encouragement and insightful discussion.

We would also like to thank anonymous reviewers who provided helpful comments on the earlier version of this paper.

\section{Quantales}\label{sec:quantale}
We first introduce \emph{quantales}, also known as \emph{complete 
idempotent semirings}.
They are the enriching (or base) categories
in the portion of enriched category theory we shall be concerned with;
in tropical mathematics, elements of a quantale are often considered as 
\emph{scalars}, and play the role similar to that of real or complex numbers in ordinary mathematics.

\begin{definition}[\cite{Mulvey_and}]\label{def:quantale}
A (unital) \defemph{quantale} $\cat{Q}$
is a complete lattice $(Q,\preceq_\cat{Q})$
equipped with a monoid structure $(Q,\one_\cat{Q},\tms_\cat{Q})$
such that the multiplication $\tms_\cat{Q}$ preserves arbitrary suprema each 
variable:
$(\bigvee_{i\in I} y_i)\tms_\cat{Q} x = \bigvee_{i\in I}(y_i\tms_\cat{Q} 
x)$ and 
$y\tms_\cat{Q} (\bigvee_{i\in I} x_i) = \bigvee_{i\in I}(y\tms_\cat{Q} x_i)$.
We often omit the subscript $\cat{Q}$ from the data of a quantale,
writing it simply as $\cat{Q}=(Q,\preceq,\one,\tms)$.
\end{definition}

Though we have introduced quantales in an order-theoretic manner, we could have 
introduced them in a more algebraic manner. 
In tropical mathematics, it is customary to start from the notion of \defemph{idempotent semiring}, namely a semiring $(Q,\zero,\one,\pls,\tms)$ with an idempotent addition:
$x\pls x = x$. 
%\begin{definition}\label{def:idem_semiring}
%An \defemph{idempotent semiring} is a semiring 
%$\cat{Q}=(Q,\zero,\one,\pls,\tms)$
%with idempotent addition: $x\pls x = x$.
%\end{definition}
Given such an idempotent semiring,
its additive part $(Q,\zero,\pls)$ is an idempotent commutative monoid, and thus 
induces a partial order $\preceq$ on $Q$ 
by $x\preceq y$ if and only if $x\pls y = y$.
The resulting poset $(Q,\preceq)$ is an 
(upper) semilattice (= poset with all finite suprema), with the least element $\zero$ and the binary supremum operation $\pls$.
Therefore quantales may also be construed as idempotent semirings with
an additional \emph{completeness} property.

Notice that in the definition of quantale,
we do {not} assume commutativity of the multiplication $\tms$ by default; 
quantales with commutative multiplication are said to be 
\defemph{commutative}.
Although many examples of quantales we shall consider below are
in fact commutative, we work within the more general noncommutative
setting in order to underline a certain parallelism between the structure of a 
quantale and that of \emph{matrices} (of arbitrary dimensions) with 
coefficients in a quantale (Section~\ref{sec:matrices}).
For matrices the multiplication (or composition) is not commutative; indeed, 
{commutativity of multiplication does not even make sense}. 

\medskip

The notion of \emph{adjunction} between posets is central to our approach (see also \cite{CGQ_duality,Belohlavek_Konecny}).
Recall that given posets $(L,\preceq)$ and $(L',\preceq')$,
two functions $f\colon L\longrightarrow L'$
and $u\colon L'\longrightarrow L$ are said to form an \defemph{adjunction}
(or \emph{Galois connection})
if for any $l\in L$ and $l'\in L'$,
\begin{equation}\label{eqn:adjointness_poset}
f(l)\preceq' l' \iff l\preceq u(l')
\end{equation}
holds. 
We call $f$ the \defemph{left adjoint} of $u$ and $u$ the \defemph{right 
adjoint} of $f$,
and write them as $f\dashv u$.
The adjointness relation (\ref{eqn:adjointness_poset}) is powerful enough to 
determine each of the functions $f$ and $u$ from the other, and force both of 
them to be \emph{monotone} functions (see \cite{Street_core}).

We record the following well-known fact, to which we shall resort frequently in 
the sequel.
\begin{proposition}[{See e.g., \cite[Proposition 7.34]{Davey_Priestley}}]\label{prop:adj_funct_thm}
Let $(L,\preceq_L)$ be a complete lattice and $(P,\preceq_P)$ be a poset.
A function $f\colon L\longrightarrow P$ preserves arbitrary suprema
if and only if there exists a function $u\colon P\longrightarrow L$
such that $f\dashv u$.
\end{proposition}

As the first application of Proposition~\ref{prop:adj_funct_thm}, 
observe that in any quantale $\cat{Q}=(Q,\preceq,$ $\one,\tms)$
there are two \emph{residuation} operations:
for any $x\in Q$, the function $(-)\tms x\colon Q\longrightarrow Q$
preserves arbitrary suprema, and hence has a right adjoint 
$(-)\mot x\colon Q\longrightarrow Q$
called the \defemph{right extension along $x$};
similarly, for any $y\in Q$ the function 
 $y\tms (-)$ has a right adjoint $y\mto (-)$,
called the \defemph{right lifting along $y$}.
Of course, in a commutative quantale the right extensions and right liftings 
coincide.
The defining adjointness relations are:
\begin{equation}\label{eqn:residual_adjointness}
y\preceq z\mot x \iff y\tms x\preceq z \iff x\preceq y\mto z.
\end{equation}
Focusing on the leftmost and rightmost formulas of 
(\ref{eqn:residual_adjointness}), we obtain
\[
z\mot x\succeq y \iff x\preceq y\mto z.
\]
That is, $z\mot (-)\colon Q\longrightarrow Q$ (regarded as a 
function from the poset $\cat{Q}=(Q,\preceq)$ to its
dual $\cat{Q}^\op=(Q,\succeq)$) is the left 
adjoint of $(-)\mto z\colon 
Q\longrightarrow {Q}$ (from $\cat{Q}^\op$ to $\cat{Q}$)
for any $z\in Q$.
The three types of adjunctions
\begin{equation*}
\begin{tikzpicture}[baseline=-\the\dimexpr\fontdimen22\textfont2\relax ]
      \node (L) at (0,0) {$\cat{Q}$};
      \node (R) at (3,0) {$\cat{Q}$};
      
      \draw [->, bend left=20] (L) to node [auto,labelsize] {$(-)\tms x$} (R);
      \draw [<-, bend right=20] (L) to node [auto,labelsize,swap] {$(-)\mot x$} 
      (R);
      
      \node [rotate=-90]at (1.5,0) {$\dashv$};
\end{tikzpicture}
\qquad
\begin{tikzpicture}[baseline=-\the\dimexpr\fontdimen22\textfont2\relax ]
      \node (L) at (0,0) {$\cat{Q}$};
      \node (R) at (3,0) {$\cat{Q}$};
      
      \draw [->, bend left=20] (L) to node [auto,labelsize] {$y\tms (-)$} (R);
      \draw [<-, bend right=20] (L) to node [auto,labelsize,swap] {$y\mto (-)$} 
      (R);
      
      \node [rotate=-90]at (1.5,0) {$\dashv$};
\end{tikzpicture}
\qquad
\begin{tikzpicture}[baseline=-\the\dimexpr\fontdimen22\textfont2\relax ]
      \node (L) at (0,0) {$\cat{Q}$};
      \node (R) at (3,0) {$\cat{Q}^\op$};
      
      \draw [->, bend left=20] (L) to node [auto,labelsize] {$z\mot (-)$} (R);
      \draw [<-, bend right=20] (L) to node [auto,labelsize,swap] {$(-)\mto z$} 
      (R);
      
      \node [rotate=-90]at (1.5,0) {$\dashv$};
\end{tikzpicture}
\end{equation*}
are fundamental in theory of quantales. 
In Sections~\ref{sec:matrices} and \ref{sec:semimodule}, we shall see that similar structures arise from matrices with coefficients in a quantale, and complete semimodules over a quantale, respectively.

\medskip

We conclude this section with some examples of quantales.

\begin{example}
The \defemph{max-plus quantale} $\Rmax=([-\infty,\infty], \leq, 0, +)$.
Here, $([-\infty,\infty],\leq)$ is the poset $(\mathbb{R},\leq)$ of all real 
numbers (with the usual order)
extended with the least element $-\infty$ and the greatest element $\infty$.
The multiplication $+$ is the unique supremum-preserving extension to $[-\infty,\infty]$ of 
the ordinary addition of real numbers; the right extension/right lifting is then uniquely determined by the adjointness relation (\ref{eqn:residual_adjointness}), and is an extension of the ordinary subtraction of real numbers; see Table~\ref{table:max_plus}.
These extended addition and subtraction appear in \cite{Lawvere_state}.
The term ``max-plus'', common in tropical mathematics, derives from the fact that as an idempotent semiring, the addition and the multiplication in $\Rmax$ are the max operation and (an extension of) the ordinary addition respectively.

Sometimes in the literature the \emph{max-plus semifield} $([-\infty, \infty),-\infty , 0,\max, +)$, which is basically $\Rmax$ without the greatest element $\infty$, is used \cite{CuninghameGreen_minmax_appl,LMS_idem_fa}.
Part of the reasons to prefer it seems to lie in the fact that it is a \emph{semifield}, meaning that every non-zero (i.e., $\neq -\infty$; note that the zero in this semiring is $-\infty$) element has a multiplicative inverse (given by $x\mapsto -x$).
However, it is our impression that quantales are much more convenient mathematical objects to work with than idempotent semifields.
For instance, the underlying poset of an idempotent semifield cannot be a 
complete lattice except for trivial cases. 
\end{example}

\begin{table}[t]
\centering
\begin{tabular}{c c|c c c}
  \multicolumn{2}{c|}{\multirow{2}{*}{$y+x$}} &  & $x$ & \\
 \multicolumn{2}{c|}{} & $\infty$  & $s$       & $-\infty$ \\ 
 \hline
    & $\infty$     & $\infty$  & $\infty$  & $-\infty$ \\ 
 $y$&$t$           & $\infty$  &  $t+s$    & $-\infty$ \\
    &$-\infty$     & $-\infty$ & $-\infty$ & $-\infty$ 
\end{tabular}
\qquad
\begin{tabular}{c c|c c c} 
 \multicolumn{2}{c|}{\multirow{2}{*}{$z-y$}} &  & $y$ & \\
 \multicolumn{2}{c|}{} & $\infty$  & $t$       & $-\infty$ \\ 
 \hline
 &$\infty$      & $\infty$  & $\infty$  & $\infty$ \\ 
$z$ &$u$           & $-\infty$  &  $u-t$    & $\infty$ \\
 &$-\infty$     & $-\infty$ & $-\infty$ & $\infty$ 
\end{tabular}
\caption{The operation tables for $y+x\ (=y\tms x)$ and $z-y\ (= z\mot y = y\mto z)$ in $\Rmax$. $s,t$ and $r$ denote real numbers.}
\label{table:max_plus}
\end{table}

\begin{example}
The \defemph{min-plus quantale} $\Rmin=([-\infty,\infty], \geq, 0, +)$.
The underlying poset $([-\infty,\infty],\geq)$ is the poset $(\mathbb{R},\geq)$ of all real numbers ordered by the \emph{opposite} of the usual order, extended with the \emph{greatest} element $-\infty$ and the \emph{least} element $\infty$.
The extension of $+$ is again uniquely determined by the requirement that it should be part of a quantale structure.

In fact, the quantale $\Rmin$ is isomorphic to $\Rmax$ under the mapping $x\mapsto -x$. 
So they are just different ways of presenting the same structure.
However, note that this does not mean that the operation tables are the \emph{same} for $\Rmax$ and $\Rmin$; for instance, we have $-\infty + \infty = -\infty$
in $\Rmax$, but have $-\infty + \infty =\infty$ in $\Rmin$.
\end{example}

\begin{example}
The \defemph{nonnegative min-plus quantale} (also called the \emph{Lawvere quantale}) $\nonnegRmin=([0,\infty], \geq, 0,+)$.
This is $\Rmin$ restricted to the nonnegative (i.e., $\geq 0$) part.
The $+$ operation is the restriction of $+$ for $\Rmin$ to $[0,\infty]$, which is anyway the only sensible extension of the ordinary addition. 
The right extension/right lifting is given by an extension of the
\emph{truncated subtraction $\dotminus$}, defined by
$u \dotminus t = \max\{u-t, 0\}$ for nonnegative real numbers $t$ and $u$.
This quantale is introduced in \cite{Lawvere_metric} for a category-theoretic approach to the theory of metric spaces. 

Note that by restricting the isomorphism between $\Rmax$ and $\Rmin$, $\nonnegRmin$ is isomorphic to the \defemph{nonpositive max-plus quantale} (in the obvious sense), but it is not isomorphic to ``the nonnegative max-plus quantale''; indeed, there is no quantale that deserves this name.
It is impossible to endow a quantale structure on the poset $([0,\infty],\leq)$ in such a way that its multiplication is given by an extension of the ordinary addition $+$, because the quantale axioms would force $0+x=0$ for all $x\in [0,\infty]$.
\end{example}

\begin{example}
One can also consider the ``discrete'' variants of the quantales $\Rmax$ ($\cong \Rmin$) or $\nonnegRmin$, replacing the real numbers by the integers.
The discrete variant of $\Rmin$ is used in \cite{Fujii_BThesis} in connection to discrete convex analysis \cite{Murota_DCA}.
\end{example}

\begin{example}
	The \defemph{nonnegative min-max quantale} $\nonnegRminmax=([0, \infty],\geq, 0, \max)$.
	Its underlying poset $([0,\infty],\geq)$ is the same as that of $\nonnegRmin$. We take the binary max operation with respect to the usual ordering $\leq$, namely the binary infimum operation with respect to $\geq$, as the multiplication. The right extension/right lifting is given by 
	\[
	z\mot y=y\mto z=\begin{cases}
		0 \ \text{ if }y\geq z\\
		z \ \text{ otherwise.}
	\end{cases}
	\]
	This quantale is related to (a generalisation of) ultrametric spaces.
	
	We remark that more generally, any \defemph{locale}, i.e., a complete lattice in which the binary infimum operation $\wedge$ satisfies the infinitary distributive law $(\bigvee_{i\in I} y_i)\wedge x=\bigvee_{i\in I}(y_i\wedge x)$, acquires a quantale structure with $\wedge$ as the multiplication;
indeed quantales were first introduced as a quantum theoretic generalisation of locales \cite{Mulvey_and}.
The poset $([0,\infty],\geq)$, or more generally any totally ordered complete lattice, is a locale.
\end{example}

\begin{example}
The \defemph{truth value quantale} (or the \emph{two-element quantale}) $\TV=(\{\bot,\top\}, 
\vdash ,\top, \wedge)$.
The underlying poset of this quantale consists of $\top$ for ``truth'' and $\bot$ for ``falsity'', ordered by the entailment relation $\vdash$, so that $\bot \vdash \top$.
The monoid structure is given by conjunction $\wedge$.
This is also used in \cite{Lawvere_metric}.
\end{example}

\begin{example}
Let $\cat{M}=(M,e,\cdot)$ be a monoid. The
\defemph{free quantale generated by $\cat{M}$}
is $\Pow \cat{M}=(\Pow M,\subseteq,\{e\}, \cdot)$,
where $\Pow M$ is the power set of $M$ and the multiplication $\cdot$ on $\Pow 
M$ is the unique supremum-preserving extension of the original multiplication on $M$, which is given by
\[
A\cdot B = \{\,a\cdot b\mid a\in A, b\in B\,\}
\]
for all $A,B\in\Pow M$.
Unlike the previous examples, this quantale is not commutative unless $\cat{M}$ is.
It has been used in connection to formal language theory, both from the categorical side \cite{Rosenthal_quantaloid_automata} as well as from the tropical side \cite{Pin_tropical}.
\end{example}

\begin{example}\label{ex:quantale_binary_rel}
Let $A$ be a set. The poset $(\Pow (A\times A),\subseteq)$
of all binary relations on $A$ admits a quantale structure
$(\Pow(A\times A),\subseteq, I_{A},\circ)$, where $I_A$ denotes the diagonal relation on $A$ and $\circ$ denotes composition of relations.
This quantale is not commutative in general. 
In fact, this is obtained from $\TV$ via a general construction of \emph{matrix quantale}; see the next section.
\end{example}

\section{Matrices}\label{sec:matrices}
In this section we study matrices over a quantale.
\begin{definition}
Let $\cat{Q}$ be a quantale, and $A$ and $B$ be (possibly infinite) sets.
A \defemph{$\cat{Q}$-matrix} (or simply a \defemph{matrix}) \defemph{from $A$ to $B$} is a $(B\times A)$-indexed
family of elements of $Q$.
If  $X = (X_{b,a})_{b\in B, a\in A}$ is a $\cat{Q}$-matrix from $A$ to $B$,
we denote it by $X\colon A\pto B$.
\end{definition}

In particular, if $A=\{\,1,\dots,m\,\}$ and $B=\{\,1,\dots,n\,\}$
for some natural numbers $m$ and $n$, then a $\cat{Q}$-matrix
$X\colon A\pto B$ is nothing but an $(n\times m)$-matrix with coefficients in $\cat{Q}$, which we may write as
\begin{equation}\label{eqn:matrix_picture}
X=
\left[
\begin{tabular}{c c c c} 
 $X_{1,1}$ & $X_{1,2}$ & $\cdots$ & $X_{1,m}$ \\
 $X_{2,1}$ & $X_{2,2}$ & $\cdots$ & $X_{2,m}$ \\
 $\vdots$  & $\vdots$  & $\ddots$ & $\vdots$  \\
 $X_{n,1}$ & $X_{n,2}$ & $\cdots$ & $X_{n,m}$ \\
\end{tabular}
\right]
\end{equation}
as usual.

For any set $A$ we have the \defemph{identity matrix} $I_A$ 
\[
(I_A)_{a,a'} = \begin{cases}
\one_\cat{Q} & \text{ if } a= a',\\
\zero_\cat{Q} &\text{ otherwise,}
\end{cases}
\]
where $\zero_\cat{Q}$ denotes the least element in $(Q,\preceq_\cat{Q})$.
Given matrices $X\colon A\pto B$ and $Y\colon B\pto C$,
we may define their \defemph{composition} $Y\tms X\colon A\pto C$ as follows:
\begin{equation}\label{eqn:comp_matrix}
(Y\tms X)_{c,a} = \bigvee_{b\in B}(Y_{c,b}\tms_\cat{Q} X_{b,a}).
\end{equation}
Since the supremum $\bigvee$ corresponds to (possibly infinitary) addition in a quantale,
\eqref{eqn:comp_matrix} may be seen as the tropical analogue of the usual matrix multiplication formula.
It can be easily checked that the composition operation satisfies the familiar laws, namely
\begin{equation}\label{eqn:comp_matr_unit_assoc}
	X\tms I_A =X,\quad I_B\tms X =X\quad\text{and}\quad
	Z\tms(Y\tms X)=(Z\tms Y)\tms X
\end{equation}
for all $X\colon A\pto B$, $Y\colon B\pto C$ and $Z\colon C\pto D$.

We put a partial order $\preceq_{A,B}$ on the set $\Mat{\cat{Q}}(A,B)$
of all $\cat{Q}$-matrices from $A$ to $B$ in an entrywise manner, namely
$X\preceq_{A,B} X'$ if and only if for all $b\in B$ and $a\in A$, $X_{b,a}\preceq_\cat{Q}
X'_{b,a}$.
The poset $(\Mat{\cat{Q}}(A,B),\preceq_{A,B})$ is isomorphic 
to the $(B\times A)$-fold product of $(Q,\preceq_\cat{Q})$, and hence is again 
a complete lattice. 
We sometimes omit the subscripts $A,B$ in $\preceq_{A,B}$.
One easily verifies that the composition operation
preserves arbitrary suprema in each variable:
\begin{equation}\label{eqn:comp_matr_pres_sup}
(\bigvee_{i\in I} Y_i)\tms X = \bigvee_{i\in I}(Y_i\tms X)
\quad\text{and}\quad
Y\tms (\bigvee_{i\in I} X_i) = \bigvee_{i\in I}(Y\tms X_i)
\end{equation}
for all $X,X_i\colon A\pto B$ and $Y,Y_i\colon B\pto C$.

Notice the evident parallelism of the structure of a quantale
(Definition~\ref{def:quantale}) and that of matrices.
We see in particular that for any set $A$, the set of all matrices of type $A\pto A$ admits a natural structure of a quantale, forming a \defemph{matrix quantale} (cf.~Example~\ref{ex:quantale_binary_rel}).
Indeed, taking into account the non-square matrices as well, we may summarise equations \eqref{eqn:comp_matr_unit_assoc} and \eqref{eqn:comp_matr_pres_sup}
by saying that for any quantale $\cat{Q}$, $\cat{Q}$-matrices
(of arbitrary dimensions) form a \emph{quantaloid}
$\Mat{\cat{Q}}$.
We do not need a formal definition of quantaloid in this paper; informally, it is a ``many-object version'' of quantale, 
just as groupoids are groups with many objects. See e.g., \cite{Rosenthal_quantaloid_automata,Stubbe_quantaloid_dist,Shen_Zhang_Isbell_Kan}.
%In fact, after appropriate modifications, 
%most results in this paper would continue to hold even if our base quantale $\cat{Q}$
%were a quantaloid.
%One immediate benefit of such generalisation is that we can state and prove a universal characterisation of the $\Mat{(-)}$ construction, as well as its \emph{idempotency}
%$\Mat{(\Mat{\cat{Q}})}\simeq \Mat{\cat{Q}}$;
%see \cite{CJSV_modulated,Stubbe_quantaloid_dist}.

\medskip

The structural similarity of a quantale and matrices over it
suggests the possibility of extending
the constructions in the former, in particular that of residuation, to the latter.
This is indeed the case, and to do so we need no more than
Proposition~\ref{prop:adj_funct_thm}.
For any $\cat{Q}$-matrix $X\colon A\pto B$ and a set $C$, the first of the equations (\ref{eqn:comp_matr_pres_sup}) says that the function
\[
(-)\tms X\colon \Mat{\cat{Q}}(B,C)\longrightarrow \Mat{\cat{Q}}(A,C)
\]
preserves arbitrary suprema. Hence we obtain its right adjoint 
\[
(-)\mot X\colon \Mat{\cat{Q}}(A,C)\longrightarrow\Mat{\cat{Q}}(B,C),
\]
called the \defemph{right extension along $X$}.
Similarly, for any $\cat{Q}$-matrix $Y\colon B\pto C$ and a set $A$,
the function
\[
Y\tms (-)\colon \Mat{\cat{Q}}(A,B)\longrightarrow \Mat{\cat{Q}}(A,C)
\]
has a right adjoint 
\[
Y\mto (-)\colon \Mat{\cat{Q}}(A,C)\longrightarrow\Mat{\cat{Q}}(A,B),
\]
called the \defemph{right lifting along $Y$}.

Analogously to (\ref{eqn:residual_adjointness}),
the adjointness relation
\begin{equation}\label{eqn:matr_resid_adjointness}
Y\preceq_{B,C} Z\mot X\iff Y\tms X\preceq_{A,C} Z\iff
X\preceq_{A,B} Y\mto Z
\end{equation}
holds for any triple of matrices $X\colon A\pto B$, $Y\colon B\pto C$
and $Z\colon A\pto C$, implying the third type of adjunction
\begin{equation}\label{eqn:matr_resid_adjoint_diagram}
\begin{tikzpicture}[baseline=-\the\dimexpr\fontdimen22\textfont2\relax ]
      \node (L) at (0,0) {$\Mat{\cat{Q}}(A,B)$};
      \node (R) at (5.5,0) {$\Mat{\cat{Q}}(B,C)^\op$};
      
      \draw [->, transform canvas={yshift=6}] (L) to node [auto,labelsize] {$Z\mot (-)$} (R);
      \draw [<-, transform canvas={yshift=-6}] (L) to node [auto,labelsize,swap] {$(-)\mto Z$} 
      (R);
      
      \node [rotate=-90]at (2.6,0) {$\dashv$};
\end{tikzpicture}
\end{equation}
(sometimes called the \emph{Isbell adjunction} \cite{Shen_Zhang_Isbell_Kan}).
Note that, since elements of $Q$ can be identified with
$\cat{Q}$-matrices of type $1\pto 1$, where $1 = \{\ast\}$
is a singleton,
(\ref{eqn:residual_adjointness}) may be regarded as a special case of (\ref{eqn:matr_resid_adjointness})
where $A=B=C=1$. 

\medskip

Let us now take a closer look at the operations of 
right extension and right lifting of matrices.
Given matrices $X\colon A\pto B$ and $Z\colon A\pto C$,
the right extension of $Z$ along $X$ is a matrix $Z\mot X\colon B\pto C$.
First observe that we have 
\begin{equation}\label{eqn:counit_X_Z}
	(Z\mot X)\tms X\preceq_{A,C}Z,
\end{equation}
since by (\ref{eqn:matr_resid_adjointness}) this is equivalent to $Z\mot X\preceq_{B,C} Z\mot X$, which holds trivially.
\begin{comment}
Adopting notation in 2-category theory, we express the inequality (\ref{eqn:counit_X_Z}) by the diagram
\begin{equation}\label{eqn:counit_diag_X_Z}
\begin{tikzpicture}[baseline=-\the\dimexpr\fontdimen22\textfont2\relax ]
	  \node (L) at (0,-1) {$A$};
      \node (M) at (2,1) {$B$};
      \node (R) at (4,-1) {$C$,};
      
      \draw [pto] (L) to node [auto,labelsize] {$X$} (M);
      \draw [pto] (M) to node [auto,labelsize] {$Z\mot X$} (R);
      \draw [pto] (L) to node [auto,labelsize,swap] {$Z$} (R);
      
      \tzdarrow{2}{-0.2}
\end{tikzpicture}
\end{equation}
using the double arrow $\Rightarrow$ to denote $\preceq$.
\end{comment}
In more detail, the adjointness relation (\ref{eqn:matr_resid_adjointness}) says that 
$Z\mot X$ is the largest matrix of type $B\pto C$ with this property.
Namely, it is characterised by (\ref{eqn:counit_X_Z}) together with
\begin{equation*}
\text{for any }Y\colon B\pto C,\ Y\tms X\preceq_{A,C}Z
\begin{comment}
\begin{tikzpicture}[baseline=-\the\dimexpr\fontdimen22\textfont2\relax ]
	  \node (L) at (0,-1) {$A$};
      \node (M) at (2,1) {$B$};
      \node (R) at (4,-1) {$C$};
      
      \draw [pto] (L) to node [auto,labelsize] {$X$} (M);
      \draw [pto] (M) to node [auto,labelsize] {$Y$} (R);
      \draw [pto] (L) to node [auto,labelsize,swap] {$Z$} (R);
      
      \tzdarrow{2}{-0.2}
\end{tikzpicture}
\end{comment}
\text{ implies }
Y\preceq_{B,C}Z\mot X.
\begin{comment}
\begin{tikzpicture}[baseline=-\the\dimexpr\fontdimen22\textfont2\relax ]
      \node (L) at (0,0) {$B$};
      \node (R) at (3,0) {$C$.};
      
      \draw [pto, bend left=30] (L) to node [auto,labelsize] {$Y$} (R);
      \draw [pto, bend right=30] (L) to node [auto,labelsize,swap] {$Z\mot X$} 
      (R);
      
      \tzdarrow{1.5}{0}
\end{tikzpicture}
\end{comment}
\end{equation*}
Explicitly, the matrix $Z\mot X$ is given by
\begin{equation}\label{eqn:matr_ext_fml}
	(Z\mot X)_{c,b} = \bigwedge_{a\in A}(Z_{c,a}\mot X_{b,a}).
\end{equation}
The formula (\ref{eqn:matr_ext_fml}) may be verified by showing the adjointness relation (\ref{eqn:matr_resid_adjointness}) directly.
Similarly,
the right lifting $Y\mto Z$ is abstractly characterised by the properties
\begin{equation*}
Y\tms (Y\mto Z)\preceq_{A,C}Z
\begin{comment}
\begin{tikzpicture}[baseline=-\the\dimexpr\fontdimen22\textfont2\relax ]
	  \node (L) at (0,-1) {$A$};
      \node (M) at (2,1) {$B$};
      \node (R) at (4,-1) {$C$};
      
      \draw [pto] (L) to node [auto,labelsize] {$Y\mto Z$} (M);
      \draw [pto] (M) to node [auto,labelsize] {$Y$} (R);
      \draw [pto] (L) to node [auto,labelsize,swap] {$Z$} (R);
      
      \tzdarrow{2}{-0.2}
\end{tikzpicture}
\end{comment}
\end{equation*}
and
\begin{equation*}
\text{for any }X\colon A\pto B,\ 
Y\tms X\preceq_{A,C}Z
\begin{comment}
\begin{tikzpicture}[baseline=-\the\dimexpr\fontdimen22\textfont2\relax ]
	  \node (L) at (0,-1) {$A$};
      \node (M) at (2,1) {$B$};
      \node (R) at (4,-1) {$C$};
      
      \draw [pto] (L) to node [auto,labelsize] {$X$} (M);
      \draw [pto] (M) to node [auto,labelsize] {$Y$} (R);
      \draw [pto] (L) to node [auto,labelsize,swap] {$Z$} (R);
      
      \tzdarrow{2}{-0.2}
\end{tikzpicture}
\end{comment}
\text{ implies }
X\preceq_{A,B} Y\mto Z,
\begin{comment}
\begin{tikzpicture}[baseline=-\the\dimexpr\fontdimen22\textfont2\relax ]
      \node (L) at (0,0) {$A$};
      \node (R) at (3,0) {$B$,};
      
      \draw [pto, bend left=30] (L) to node [auto,labelsize] {$X$} (R);
      \draw [pto, bend right=30] (L) to node [auto,labelsize,swap] {$Y\mto Z$} 
      (R);
      
      \tzdarrow{1.5}{0}
\end{tikzpicture}
\end{comment}
\end{equation*}
whereas concretely it is given by the formula
\begin{equation}\label{eqn:right_lifting_formula}
	(Y\mto Z)_{b,a} = \bigwedge_{c\in C}(Y_{c,b}\mto Z_{c,a}).
\end{equation}

\medskip

We now introduce miscellaneous notions on matrices.
For any matrix $X\colon A\pto B$, its \defemph{transpose}
$X^\tr\colon B\pto A$ is defined by $X^\tr_{a,b}=X_{b,a}$.
Recall that $1=\{\ast\}$ is a singleton. 
As already mentioned, matrices of type $1\pto 1$ can be identified with elements of the base quantale $\cat{Q}$, and thus are also called \defemph{scalars}.
If $X\colon 1\pto 1$ is a scalar, then by abuse of notation we also denote by $X$ its unique entry $X_{\ast,\ast}\in Q$. For any set $A$, matrices of type $1\pto A$ are \defemph{column vectors of size $A$}, and dually 
those of type $A\pto 1$ are \defemph{row vectors of size $A$}; cf.~(\ref{eqn:matrix_picture}). 
Note that a column vector $X\colon 1\pto A$ and a row vector $Y\colon A\pto 1$ of the same size compose, 
and give rise to a scalar 
\begin{equation*}
	Y\tms X=\bigvee_{a\in A}(Y_{\ast,a}\tms X_{a,\ast}),
\end{equation*}
which may be thought of as the \emph{scalar product} of $Y$ and $X$.
For any matrix $X\colon A\pto B$ and $a\in A$, the 
\defemph{$a$-th column vector of $X$} is a column vector $X_{-,a}\colon 1\pto B$ of size $B$ defined by $(X_{-,a})_{b,\ast}=X_{b,a}$; dually, for any $b\in B$ we have the \defemph{$b$-th row vector of $X$} $X_{b,-}=((X^\tr)_{-,b})^\tr$.

\begin{example}\label{ex:Legendre_Fenchel}
Willerton~\cite{Willerton_Legendre} observes that the 
 Legendre--Fenchel transform (see e.g., \cite[Section~12]{Rockafellar_ca}) is an instance of right extension/right lifting of matrices.
Recall that the Legendre--Fenchel transform on a finite-dimensional real vector space $\mathbb{R}^n$ maps any function $f\colon\mathbb{R}^n\longrightarrow [-\infty,\infty]$ to its \defemph{conjugate}, which is a function $f^\ast$ of type $(\mathbb{R}^n)^\ast\longrightarrow[-\infty,\infty]$, 
where $(\mathbb{R}^n)^\ast =\mathbb{R}^n$ is the dual vector space of the original $\mathbb{R}^n$.
Using the \defemph{pairing function} $\pairing{-}{-}$ defined by $\pairing{p}{v}=\sum_{i=1}^n p_iv_i$ for each $p=(p_1,\dots,p_n)\in (\mathbb{R}^n)^\ast$ and 
$v=(v_1,\dots,v_n)\in \mathbb{R}^n$, the formula for the conjugate $f^\ast$ is given by 
\begin{equation}\label{eqn:Legendre_Fenchel}
f^\ast(p)=\sup_{v\in \mathbb{R}^n}\{\pairing{p}{v}-f(v)\}
\end{equation}
for each $p\in (\mathbb{R}^n)^\ast$.

To understand this construction via matrices, we first regard the pairing function $\pairing{-}{-}$ as an $\Rmin$-matrix $\pairing{-}{-}\colon \mathbb{R}^n\pto (\mathbb{R}^n)^\ast$. A function $f$ of type $\mathbb{R}^n\longrightarrow[-\infty,\infty]$ may be identified with an $\Rmin$-matrix $f\colon \mathbb{R}^n\pto 1$, and likewise a function $g$ of type $(\mathbb{R}^n)^\ast\longrightarrow[-\infty,\infty]$ with an $\Rmin$-matrix $g\colon 1\pto (\mathbb{R}^n)^\ast$. Then we have $f^\ast=\pairing{-}{-}\mot f$; see (\ref{eqn:matr_ext_fml}). Note that infima in $\Rmin$ (denoted by $\bigwedge$ in (\ref{eqn:matr_ext_fml})) amount to suprema with respect to the usual order (denoted by $\sup$ in (\ref{eqn:Legendre_Fenchel})).

The Legendre--Fenchel transform works in the converse direction as well, mapping a function $g\colon (\mathbb{R}^n)^\ast \longrightarrow[-\infty,\infty]$ to its conjugate $g^\ast\colon \mathbb{R}^n\longrightarrow[-\infty,\infty]$ given by essentially the same formula:
\[
g^\ast(v)=\sup_{p\in(\mathbb{R}^n)^\ast}\{\pairing{p}{v}-g(p)\}.
\]
Of course, this can be expressed as $g^\ast=g\mto \pairing{-}{-}$ using matrices.

The Legendre--Fenchel transform is an important tool in convex analysis. 
One reason for this importance is that for any function $f\colon \mathbb{R}^n\longrightarrow [-\infty,\infty]$, $(f^\ast)^\ast$ is its \defemph{convex envlope}, i.e., the least (with respect to the pointwise order $\leq$) lower-semicontinuous convex function above $f$. 
In particular, lower-semicontinuous convex functions can be characterised as the ``fixed points'' of the Legendre--Fenchel transform; we shall revisit this point in the next section.
\end{example}

\section{The Isbell hull of a matrix}\label{sec:Isbell_hull}
In this section we introduce the notion of Isbell hull of a matrix.
It incorporates many important constructions in tropical mathematics and related fields, such as the tropical polytope generated by a set of points \cite{Develin_Sturmfels_trop_conv}, the directed tight span of a (generalised) metric space \cite{Hirai_Koichi_tight_span,Kemajou_Kunzi_Otafudu_Isbell_hull}, and the set of lower-semicontinuous convex functions on $\mathbb{R}^n$.

\begin{definition}[{\cite[Section 4]{Shen_Zhang_Isbell_Kan}}]\label{def:Isbell_hull}
Let $\cat{Q}$ be a quantale and $Z\colon A\pto C$ be a $\cat{Q}$-matrix.
The \defemph{Isbell hull of $Z$}, denoted by $\Isb(Z)$,
is the set of all pairs $(X,Y)$ of $\cat{Q}$-matrices $X\colon A\pto 1$ and
$Y\colon 1\pto C$ such that both $Y=Z\mot X$ and $X=Y\mto Z$ hold.
%in other words, such that the following is {simultaneously} a right extension and a right lifting diagram:
%\begin{equation}\label{eqn:Isb_hull_X_Y}
%\begin{tikzpicture}[baseline=-\the\dimexpr\fontdimen22\textfont2\relax ]
%	  \node (L) at (0,-1) {$A$};
%      \node (M) at (2,1) {$1$};
%      \node (R) at (4,-1) {$C$.};
%      
%      \draw [pto] (L) to node [auto,labelsize] {$X$} (M);
%      \draw [pto] (M) to node [auto,labelsize] {$Y$} (R);
%      \draw [pto] (L) to node [auto,labelsize,swap] {$Z$} (R);
%      
%      \tzdarrow{2}{-0.2}
%\end{tikzpicture}
%\end{equation}
\end{definition}
In this section we treat the Isbell hull of a matrix only as 
a \emph{set} of certain pairs of matrices, but in fact it has rich structure, as we shall see in Section~\ref{sec:semimodule}.
The term ``Isbell hull'' is already used in \cite{Kemajou_Kunzi_Otafudu_Isbell_hull} to refer to a special case of Definition~\ref{def:Isbell_hull}.

The Isbell hull construction, like many important constructions, may be viewed from several different angles.
The first alternative view is provided by the following.

\begin{proposition}\label{prop:Isbell_hull_alternatively}
Let $Z\colon A\pto C$ be a $\cat{Q}$-matrix. A pair $(X,Y)\in\Mat{\cat{Q}}(A,1)\times \Mat{\cat{Q}}(1,C)$ belongs to $\Isb(Z)$
if and only if it is a maximal pair under-approximating $Z$,
in the sense that (i) $Y\tms X\preceq Z$, and (ii) if $(X',Y')\in \Mat{\cat{Q}}(A,1)\times \Mat{\cat{Q}}(1,C)$ satisfies $Y'\tms X'\preceq Z$, $X\preceq X'$
and $Y\preceq Y'$, then $X=X'$ and $Y=Y'$.
\end{proposition}
\begin{proof}
	Suppose $(X,Y)\in \Isb(Z)$.
	Given $(X',Y')$ with $Y'\tms X' \preceq Z$, 
	$X\preceq X'$ implies $Y'\preceq Z\mot X'\preceq Z\mot X=Y$;
	similarly, $Y\preceq Y'$ implies $X'\preceq X$.
	
	Conversely, suppose that $(X,Y)$ satisfies conditions (i) and (ii).
	Then the pair $(X,Z\mot X)$ satisfies 
	$(Z\mot X)\tms X\preceq Z$, $X\preceq X$ and, by (i), $Y\preceq Z\mot X$. So by (ii) we conclude $Y=Z\mot X$.
	Similarly, using the pair $(Y\mto Z,Y)$ we see that $X=Y\mto Z$.
\end{proof}

The Isbell hull of a matrix can be seen as an instance of the set of fixed points of an adjunction.
Suppose that 
\begin{equation*} 
\begin{tikzpicture}[baseline=-\the\dimexpr\fontdimen22\textfont2\relax ]
      \node (L) at (0,0) {$\cat{L}$};
      \node (R) at (3,0) {$\cat{L'}$};
      
      \draw [->, bend left=20] (L) to node [auto,labelsize] {$f$} (R);
      \draw [<-, bend right=20] (L) to node [auto,labelsize,swap] {$u$} 
      (R);
      
      \node [rotate=-90]at (1.5,0) {$\dashv$};
\end{tikzpicture}
\end{equation*}
is an adjunction between posets $\cat{L}=(L,\preceq)$ and $\cat{L'}=(L',\preceq')$.
By a \defemph{fixed point of the adjunction $f\dashv u$} we mean a pair $(l,l')\in L\times L'$ such that $l'=f(l)$ and $l=u(l')$.
The set of all fixed points of $f\dashv u$ is denoted by $\Fix(f\dashv u)$.
For any matrix $Z\colon A\pto C$, clearly $\Isb(Z)$ is $\Fix(Z\mot(-)\dashv 
(-)\mto Z)$ for the adjunction (\ref{eqn:matr_resid_adjoint_diagram}) with 
$B=1$.

We recall some well-known properties of $\Fix(f\dashv u)$.
By definition, given an element $(l,l')\in \Fix(f\dashv u)$, each of $l$ and $l'$ is determined from the other. 
This means that both projections $\pi_1\colon \Fix(f\dashv u)\longrightarrow L$ and $\pi_2\colon \Fix(f\dashv u)\longrightarrow L'$, mapping $(l,l')$ to $l$ and $l'$ respectively, are injective.
Given elements $(l_1,l'_1)$ and $(l_2,l'_2)$ of $\Fix(f\dashv u)$, we have $l_1\preceq l_2$ if and only if $l'_1\preceq' l'_2$, because the former is equivalent to $l_1\preceq u(l'_2)$ and the latter to $f(l_1)\preceq' l'_2$.
Hence we obtain a natural partial order on $\Fix(f\dashv u)$ induced by either of the projections, turning it into a poset $\Fixpos(f\dashv u)$.
The projections $\pi_1$ and $\pi_2$ are monotone by definition and moreover they have adjoints:
\begin{equation*}
\begin{tikzpicture}[baseline=-\the\dimexpr\fontdimen22\textfont2\relax ]
      \node (L) at (0,1) {$\cat{L}$};
      \node (R) at (4,1) {$\cat{L'}$};
      \node (B) at (2,-1){$\Fixpos(f\dashv u)$.};
      
      \draw [->, bend left=15] (L) to node [auto,labelsize] {$f$} (R);
      \draw [<-, bend right=15] (L) to node [auto,labelsize,swap] {$u$} (R);
       \draw [<-, bend left=18] (L) to node [midway,fill=white,labelsize] {$\pi_1$} (B);
       \draw [->, bend right=15] (L) to node [auto,labelsize,swap] {$g$} (B);
       \draw [->, bend left=15] (R) to node [auto,labelsize] {$v$} (B);
       \draw [<-, bend right=18] (R) to node [midway,fill=white,labelsize,swap] {$\pi_2$} (B);      
      
      \node [rotate=-90]at (2,1) {$\dashv$};
      \node [rotate=45]at (1,0) {$\dashv$};
      \node [rotate=-45]at (3,0) {$\dashv$};
\end{tikzpicture}
\end{equation*}
Here, the adjoints $g$ and $v$ are given by $g(l)=(uf(l), f(l))$ and $v(l')=(u(l'), fu(l'))$.
Observe in particular that the image of $\pi_1$ (which may be identified with $\Fix(f\dashv u)$ by injectivity of $\pi_1$) coincides with the image of $u$.
Infima and suprema in $\Fixpos(f\dashv u)$ are easy to describe. 
Given a family $((l_i,l'_i))_{i\in I}$ of elements of $\Fix(f\dashv u)$, $\bigwedge_{i\in I}(l_i,l'_i)\in\Fix(f\dashv u)$ exists if and only if $\bigwedge_{i\in I}l_i\in L$ exists, and in this case we have $\bigwedge_{i\in I} (l_i,l'_i) =(\bigwedge_{i\in I}l_i,f(\bigwedge_{i\in I} l_i))$; 
similarly, 
$\bigvee_{i\in I}(l_i,l'_i)=( u(\bigvee_{i\in I} l'_i), \bigvee_{i\in I} l'_i)$ whenever either side of the equality exists.
In particular, $\Fixpos(f\dashv u)$ is a complete lattice whenever either $\cat{L}$ or $\cat{L'}$ is. 

\begin{corollary}\label{cor:Isbell_hull_cor}
	Let $Z\colon A\pto C$ be a matrix.
	\begin{enumerate}
	\item The Isbell hull $\Isb(Z)$ of $Z$ admits a partial order $\preceq$ defined by $(X,Y)\preceq (X',Y')$ if and only if $X\preceq_{A,1} X'$, or equivalently if and only if $Y'\preceq_{1,C}Y$. 
	The poset $(\Isb(Z),\preceq)$ is a complete lattice.
	\item	A row vector $X\colon A\pto 1$ is in $\Isb(Z)$ (precisely: is in $\pi_1(\Isb(Z))$) if and only if there exists some column vector $Y\colon 1\pto C$ such that $X=Y\mto Z$.
	\item For any pair $(X',Y')\in \Mat{\cat{Q}}(A,1)\times \Mat{\cat{Q}}(1,C)$ under-approximating $Z$ (i.e., such that $Y'\tms X'\preceq Z$), there exists $(X,Y)\in\Isb(Z)$ with $X'\preceq X$ and $Y'\preceq Y$.
	\end{enumerate}
\end{corollary}
\begin{proof}
	We show only the third clause. Given any $X'\in\Mat{\cat{Q}}(A,1)$, the pair $((Z\mot X')\mto Z, Z\mot X')$ is in $\Isb(Z)$ and we have $X'\preceq (Z\mot X')\mto Z$. If $ Y'\in \Mat{\cat{Q}}(1,C)$ satisfies $Y'\tms X'\preceq Z$, then we also have $Y'\preceq Z\mot X'$.
\end{proof}

Let us see some examples of the Isbell hull construction. The case of tropical polytope requires some preparation and will be treated in Section~\ref{sec:semimodule}.

\begin{example}[\cite{Willerton_Legendre}]
The set of all lower-semicontinuous convex functions of type $\mathbb{R}^n\longrightarrow [-\infty,\infty]$ may be understood as the Isbell hull of the $\Rmin$-matrix $\pairing{-}{-}\colon \mathbb{R}^n\pto(\mathbb{R}^n)^\ast$ in Example~\ref{ex:Legendre_Fenchel}. In light of that example, this is simply a restatement of the classical theorem that a function $f\colon \mathbb{R}^n\longrightarrow[-\infty,\infty]$ is lower-semicontinuous convex if and only if $f=g^\ast$ for some $g\colon (\mathbb{R}^n)^\ast\longrightarrow [-\infty,\infty]$ (if and only if $f=(f^\ast)^\ast$).
\end{example}

\begin{example}
The trivial-looking case of $\cat{Q}=\TV$ is in fact a rich source of instances of naturally occurring Isbell hulls. In this case, a matrix $Z\colon A\pto C$ may be identified with a relation $Z\subseteq C\times A$ consisting exactly of those pairs $(c,a)\in C\times A$ with $Z_{c,a}=\top$. Similarly, a row vector $X\colon A\pto 1$ and a column vector $Y\colon 1\pto C$ are identified with subsets $X\subseteq A$ and $Y\subseteq C$ respectively. Then we have an adjunction
\begin{equation}  \label{eqn:adjoint_powerset}
\begin{tikzpicture}[baseline=-\the\dimexpr\fontdimen22\textfont2\relax ]
      \node (L) at (0,0) {$\Pow A$};
      \node (R) at (3.5,0) {$(\Pow C)^\op,$};
      
      \draw [->, transform canvas={yshift=6}] (L) to node [auto,labelsize] {$Z\mot (-)$} (R);
      \draw [<-, transform canvas={yshift=-6}] (L) to node [auto,labelsize,swap] {$(-)\mto Z$} 
      (R);
      
      \node [rotate=-90]at (1.6,0) {$\dashv$};
\end{tikzpicture}
\end{equation}
given explicitly by
\[
Z\mot X=\{\, c\in C\mid\forall a\in X.\, (c,a)\in Z \,\}
\]
and 
\[
Y\mto Z=\{\, a\in A\mid\forall c\in Y.\, (c,a)\in Z \,\}.
\]
In fact, every adjunction between the powerset lattices $\Pow A$ and $(\Pow C)^\op$ is of the form (\ref{eqn:adjoint_powerset})
for some relation $Z\subseteq A\times C$. (One can reconstruct a relation from such an adjunction by considering the images of singletons.)

Being a general form of adjunctions (Galois connections) between powerset lattices, this includes many classical dualities; see e.g., \cite[Section~1.1.1]{Willerton_Legendre}. This is also the framework of \emph{formal concept analysis}, in which one is given the data consisting of a set $A$ of \emph{attributes}, a set $C$ of \emph{objects} and a \emph{satisfaction relation} $Z$ specifying which object satisfies which attribute, and aims to extract \emph{concepts}, which are nothing but elements of $\Isb(Z)$. Pavlovic \cite{Pavlovic_quantitative} proposes an enriched categorical generalisation of formal concept analysis via the \emph{nucleus of a profunctor}, which is the same as the Isbell hull of a $\cat{Q}$-matrix for a general quantale $\cat{Q}$.

Another classical instance of this construction is where $A=C$ and $Z$ is a partial order on $A$. In this case, the poset $\Isbpos(Z)$ (which is a complete lattice by Corollary~\ref{cor:Isbell_hull_cor}) is known as the \emph{MacNeille completion} of the poset $(A,Z)$.
\end{example}

\begin{example}\label{ex:d_tight_span}
Willerton~\cite{Willerton_tight_span} observes that the directed tight span of a generalised metric space \cite{Hirai_Koichi_tight_span,Kemajou_Kunzi_Otafudu_Isbell_hull} is an instance of the Isbell hull construction with $\cat{Q}=\nonnegRmin$.
A \defemph{generalised metric space} $(A,d)$ consists of a set $A$ together with a function $d\colon A\times A\longrightarrow[0,\infty]$ satisfying $d(a,a)=0$ and $d(a,a')+d(a'a'')\geq d(a,a'')$ (the triangle inequality) for each $a,a',a''\in A$.
Given a generalised metric space $(A,d)$, its \defemph{directed tight span} is given by the set of all pairs of functions $(X,Y)$ where $X,Y\colon A\longrightarrow [0,\infty]$, such that $X$ and $Y$ form a minimal pair satisfying $Y(a')+X(a)\geq d(a,a')$ for each $a,a'\in A$.
In light of Proposition~\ref{prop:Isbell_hull_alternatively}, this is clearly the Isbell hull of $d$ regarded as an $\nonnegRmin$-matrix $d\colon A\pto A$.

The directed tight span of a generalised metric space has a natural $l_\infty$-type metric on it, making it into a generalised metric space as well. 
(Indeed, the directed tight span of $(A,d)$ is characterised as the unique \emph{injective} and \emph{tight extension} of $(A,d)$.)
As we shall see in Example \ref{ex:nonnegRmin_cat_gms}, the notion of generalised metric space coincides with an enriched categorical notion of \emph{$\nonnegRmin$-categories}.
We can also recover the $l_\infty$ metric on the directed tight span from the general theory, since as we shall see in Sections \ref{sec:semimodule} and \ref{sec:Q-cat}, the Isbell hull of a $\cat{Q}$-matrix always acquires a natural structure of a $\cat{Q}$-category. 
\end{example}

\section{Complete semimodules}\label{sec:semimodule}
In this section we study the notion of complete semimodule over a quantale $\cat{Q}$. It will turn out that the Isbell hull  of any $\cat{Q}$-matrix acquires such structure and conversely, any complete semimodule is isomorphic to the Isbell hull of itself (regarded as a $\cat{Q}$-category; see Section \ref{sec:Q-cat}). 

\begin{definition}[{\cite[Section~2.2]{CGQ_duality}}]\label{def:comp_semimodule}
Let $\cat{Q}=(Q,\preceq_\cat{Q},\one_\cat{Q},\tms_{\cat{Q}})$ be a quantale.
A \defemph{(left) complete semimodule over $\cat{Q}$} is a 
complete lattice $\cat{M}=(M,\preceq_{\cat{M}})$ equipped with a left action $\act_{\cat{M}} \colon Q\times M\longrightarrow M$
of the monoid $(Q,\one_\cat{Q},\tms_\cat{Q})$ on the set $M$,
such that $\act_\cat{M}$ preserves arbitrary suprema in each argument:
$(\bigvee_{i\in I}x_i)\act_\cat{M} m = \bigvee_{i\in I}(x_i\act_\cat{M} m)$ and
$x\act_\cat{M}(\bigvee_{i\in I} m_i)=\bigvee_{i\in 
I}(x\act_\cat{M} m_i)$.

A \defemph{right complete semimodule over $\cat{Q}$} is a left complete semimodule over the quantale $\cat{Q}^\rev=(Q,\preceq_\cat{Q},I_\cat{Q},\tms^\rev_\cat{Q})$, where $x\tms^\rev_\cat{Q}y=y\tms_\cat{Q} x$. 
\end{definition}

Analogously to the case of quantale, there is an algebraic variant of the above order-theoretic definition.
In general, a \emph{semimodule over a semiring $\cat{Q}$} is a commutative monoid $\cat{M}$ equipped with an action of $\cat{Q}$, which may be succinctly expressed as a semiring homomorphism from $\cat{Q}$ to the semiring $\End(\cat{M})$ of all endomorphisms on $\cat{M}$. 
If $\cat{Q}$ is an idempotent semiring then the underlying commutative monoid $\cat{M}$ of any semimodule over it is necessarily idempotent, since the existence of a semiring homomorphism $\cat{Q}\longrightarrow\End(\cat{M})$ forces $\id{\cat{M}}\in\End(\cat{M})$ to be idempotent (with respect to addition).
So a semimodule over an idempotent semiring can be equivalently given as an (upper) semilattice with an appropriate action, and Definition~\ref{def:comp_semimodule} may be viewed as a \emph{complete} version of it.

Applying Proposition~\ref{prop:adj_funct_thm}, we obtain two residuation 
operations associated with a complete semimodule 
$\cat{M}=(M,\preceq_\cat{M},\act_\cat{M})$:
for each $m\in M$, the right adjoint $(-)\mot_\cat{M} m\colon M\longrightarrow Q$ of $(-)\act_\cat{M}m\colon Q\longrightarrow M$, and for each $x\in Q$,
the right adjoint $x\mto_\cat{M}(-)\colon M\longrightarrow M$ of 
$x\act_\cat{M}(-)\colon M\longrightarrow M$.
We again have the adjointness relation
\begin{equation}\label{eqn:residual_adj_semimod}
x\preceq_\cat{Q} n\mot_\cat{M} m \iff x\act_\cat{M} m\preceq_\cat{M} n \iff 
m\preceq_\cat{M} x\mto_\cat{M} n
\end{equation}
for each $x\in Q$ and $m,n\in M$.
Adopting terminology from enriched category theory, we call the function
\[
\act_\cat{M}\colon \cat{Q}\times \cat{M}\longrightarrow\cat{M}
\]
the \defemph{copower}, the function
\[
\mot_{\cat{M}}\colon \cat{M}\times \cat{M}^\op\longrightarrow\cat{Q}
\]
the \defemph{hom-functor}, and the function
\[
\mto_{\cat{M}}\colon \cat{Q}^\op\times \cat{M}\longrightarrow\cat{M}
\]
the \defemph{power}.

\begin{proposition}[{Cf.~\cite[Section~2.3]{CGQ_duality} and \cite[Proposition 5.10]{Stubbe_quantaloid_dist}}]\label{prop:comp_semimod_duality}
	Let $\cat{Q}=(Q,\preceq_\cat{Q},\one_\cat{Q},\tms_\cat{Q})$ be a quantale and $\cat{M}=(M,\preceq_\cat{M},\act_\cat{M})$ be a left complete semimodule over $\cat{Q}$.
	Then the triple $\cat{M}^\op=(M,\succeq_\cat{M},\mto_\cat{M})$ is a right complete semimodule over $\cat{Q}$ (called the \defemph{dual} of $\cat{M}$). 
	Moreover we have $(\cat{M}^\op)^\op=\cat{M}$.
\end{proposition}
\begin{proof}
	We need to show that $\one_\cat{Q}\mto_\cat{M}m=m$, $x\mto_\cat{M}(y\mto_\cat{M}m)=(y\tms_\cat{Q}x)\mto_\cat{M} m$,
	$(\bigvee_{i\in I}x_i)\act_{\cat{M}}m=\bigwedge_{i\in I}(x_i\act_\cat{M}m)$ and $x\act_\cat{M}(\bigwedge_{i\in I}m_i)=\bigwedge_{i\in I}(x\act_\cat{M}m_i)$, where $\bigwedge$
	denotes infima in the original complete lattice $(M,\preceq_{\cat{M}})$.
	They follow from the adjointness relation (\ref{eqn:residual_adj_semimod}).
\end{proof}

\begin{example}
	Here are some immediate examples of (left or right) complete semimodules over a quantale $\cat{Q}=(Q,\preceq_\cat{Q},\one_\cat{Q},\tms_\cat{Q})$; in the following, all complete semimodules are over $\cat{Q}$ unless otherwise stated.
	First, clearly the underlying complete lattice $(Q,\preceq_\cat{Q})$ of $\cat{Q}$ acquires a left complete semimodule structure via $\tms_\cat{Q}$; we denote this left complete semimodule also by $\cat{Q}$. 
	The dual right complete semimodule $\cat{Q}^\op$ corresponding to it (Proposition~\ref{prop:comp_semimod_duality}) uses the right lifting $\mto$ as action.
	
	On the other hand, the complete lattice $(Q,\preceq_{\cat{Q}})$ also underlies the quantale $\cat{Q}^\rev=(Q,\preceq_\cat{Q},\one_\cat{Q},\tms_\cat{Q}^\rev)$, hence via $\tms_\cat{Q}^\rev$ it becomes a left complete semimodule over $\cat{Q}^\rev$, i.e., a right complete semimodule (over $\cat{Q}$). This right complete semimodule is denoted by $\cat{Q}^\rev$.
	The dual left complete semimodule $\cat{Q}^\revop=(\cat{Q}^\rev)^\op$ uses the right extension $\mot$ as action. 
	
	Because left or right complete semimodules are closed under arbitrary products, for any pair of sets $A$ and $B$ we obtain a left complete semimodule $\cat{Q}^A\times (\cat{Q}^\revop)^B$ and a right complete semimodule $(\cat{Q}^\op)^A\times (\cat{Q}^\rev)^B$, which are easily seen to be the dual of each other.
\end{example}

The symmetry of the notions of left and right complete semimodule may leave one wonder whether it is possible to give the data of a complete semimodule in an unbiased manner, choosing neither the left action (copower) nor the right action (power) as a primitive. 
This is indeed possible, by means of the hom-functor.
The resulting notion is \emph{skeletal and complete $\cat{Q}$-category}; see Section~\ref{sec:Q-cat}. 

\begin{definition}
	Let $\cat{Q}$ be a quantale and $\cat{M}=(M,\preceq_{\cat{M}},\act_{\cat{M}})$ be a left complete semimodule over $\cat{Q}$. 
	A \defemph{left complete subsemimodule of $\cat{M}$} is a subset of $M$ closed under arbitrary suprema in $\cat{M}$ and the copower $x\act_\cat{M} (-)$ by an arbitrary element $x\in Q$.
	A \defemph{right complete subsemimodule of $\cat{M}$} is a subset of $M$ closed under arbitrary infima in $\cat{M}$ and the power $x\mto_\cat{M}(-)$ by an arbitrary element $x\in Q$.
	
	We regard a left (resp.~right) complete subsemimodule itself as a left (resp.~right) complete semimodule, by the induced order and operation.
\end{definition}
Of course, a right complete subsemimodule of $\cat{M}$ is nothing but a left complete subsemimodule of the left complete semimodule $\cat{M}^\op$ over $\cat{Q}^\rev$.

The set of all left complete subsemimodules of a fixed left complete semimodule $\cat{M}=(M,\preceq_\cat{M},\act_\cat{M})$ is closed under arbitrary intersections. Hence for any subset $S\subseteq M$ there exists the least left complete subsemimodule containing $S$, the \defemph{left complete subsemimodule generated by $S$}.
Similarly we have the \defemph{right complete subsemimodule generated by $S$}.

\begin{proposition}[{\cite[Theorem 3.1]{Belohlavek_Konecny}, cf.~\cite[Theorem 3.3.9]{Elliott_thesis} }]\label{prop:Isbell_hull_as_hull}
	Let $Z\colon A\pto C$ be a matrix. The Isbell hull $\Isb(Z)$,
	considered as a subset of $\Mat{\cat{Q}}(A,1)$, is the right complete subsemimodule generated by the set $\{\,Z_{c,-}\mid c\in C\,\}\subseteq \Mat{\cat{Q}}(A,1)$ of all row vectors of $Z$, where $\Mat{\cat{Q}}(A,1)$ is regarded as the left complete semimodule $\cat{Q}^A$. 
	Dually, $\Isb(Z)$ considered as a subset of $\Mat{\cat{Q}}(1,C)$ is the left complete subsemimodule generated by $\{\,Z_{-,a}\mid a\in A\,\}\subseteq\Mat{\cat{Q}}(1,C)$,
	where $\Mat{\cat{Q}}(1,C)$ is regarded as the left complete semimodule $(\cat{Q}^\revop)^C$.
\end{proposition}
\begin{proof}
	Let us prove that $\Isb(Z)\subseteq \cat{Q}^A$ is the right complete subsemimodule generated by $\{\,Z_{c,-}\mid c\in C\,\}$. 
	Recall the second clause of Corollary~\ref{cor:Isbell_hull_cor}, claiming that $X\colon A\pto 1$ is in $\Isb(Z)$ if and only if it is of the form $X=Y\mto Z$ for some $Y\colon 1\pto C$.
	The key observation is that the equation $X=Y\mto Z$ may be regarded as an expression of $X$ by a ``tropical linear combination (in $(\cat{Q}^\op)^A$)'' of the row vectors $Z_{c,-}$. 
	More precisely, we read from the formula (\ref{eqn:right_lifting_formula}) for right lifting that we have 
	   \begin{equation}\label{eqn:X_as_linear_combination}
	   X=Y\mto Z = \bigwedge_{c\in C} (Y_{c,\ast}\mto Z_{c,-}).
	   \end{equation}
	   Observe that the column vector $Y$ plays the role of ``coefficients'' of the above tropical linear combination.
	   Any row vector $X$ as in (\ref{eqn:X_as_linear_combination}), expressible as the infimum of powers of $Z_{c,-}$, must lie in any right complete subsemimodule containing the row vectors $Z_{c,-}$. 
	   Also we see that for any $c\in C$, the row vector $Z_{c,-}$
	   itself may be expressed as a trivial linear combination, upon taking $Y$ to be 
	   \[
	   Y_{c',\ast}=\begin{cases}
	   \one_Q &\text{if } c=c',\\
	   \zero_\cat{Q} &\text{otherwise.}
	   \end{cases}	
	   \]
	
	That the set $\Isb(Z)$ of all row vectors expressible as (\ref{eqn:X_as_linear_combination}) is closed under infima and the power $x\mto(-)$ by an element $x$ of $\cat{Q}$ may be checked easily. 
	Precisely, given a family $X_i=Y_i\mto Z$, we may take new coefficient $\bigvee_{i\in I}Y_i$ to express their infimum: $\bigwedge_{i\in I} X_i= (\bigvee_{i\in I}Y_i)\mto Z$.
	 Given $X=Y\mto Z$, we have $x\mto X= (Y\tms x)\mto Z$.
\end{proof} 

As a consequence of this proposition, the Isbell hull of any $\cat{Q}$-matrix acquires the structure of a left complete $\cat{Q}$-semimodule.

\begin{example}
	By using Proposition \ref{prop:Isbell_hull_as_hull}, we can see how tropical polytopes \cite{Develin_Sturmfels_trop_conv} are (modulo some points at infinity) instances of the Isbell hull construction.
	Let $m$ and $n$ be natural numbers. In \cite{Develin_Sturmfels_trop_conv}, the tropical polytope  generated by $m$ points $Z_1,\dots, Z_m\in\mathbb{R}^n$ is defined as the set of all points $X\in \mathbb{R}^n$ which can be expressed as a tropical linear combination 
	\begin{equation}\label{eqn:tropical_linear_conb}
		X= (Y_1\odot Z_1) \oplus \dots\oplus (Y_m \odot Z_m),
	\end{equation}
	where $Y_i\in \mathbb{R}$ and $\oplus$ and $\odot$ denote the coordinate-wise minimum and the addition ($c\odot (a_1,\dots,a_n) = (c+a_1,\dots,c+a_n)$) respectively.\footnote{Strictly speaking, Develin and Sturmfels consider an additional step of projectivisation in the definition, hence our tropical polytopes are the unprojectivised version of theirs.}
	As observed in \cite{Elliott_thesis}, this is the finite part of the Isbell hull of the $\Rmax$-matrix $Z\colon n\pto m$ (where $n$ (resp. $m$) denotes an $n$-element (resp. $m$-element) set) obtained from the coordinates of the points $Z_1,\dots,Z_m$. 
	Indeed, any expression of the form (\ref{eqn:tropical_linear_conb}) may be expressed as $X=Y\mto Z$ for some $\Rmax$-matrix $Y\colon 1\pto m$ with finite (i.e., $\neq \pm \infty$) entries and vice versa.
	
	In \cite{Develin_Sturmfels_trop_conv} the following interesting duality result has been established: there is a bijective correspondence between (the combinatorial types of) \emph{tropical polytopes in $\mathbb{R}^n$ generated by $m$ points} and \emph{tropical polytopes in $\mathbb{R}^m$ generated by $n$ points}. From our point of view, this result can be understood abstractly from the fact that for an $\Rmax$-matrix $Z$, $\Isb(Z)$ and $\Isb(Z^\tr)$ are canonically isomorphic; this follows from the commutativity of the quantale $\Rmax$.
\end{example}

%We have:
%\begin{itemize}
%\item $(p\mot_\cat{M}n)\tms_\cat{Q} (n\mot_\cat{M} m)\preceq_\cat{Q} 
%p\mot_\cat{M} m$
%(Because: $(p\mot_\cat{M}n)\tms_\cat{Q} (n\mot_\cat{M} m)\tms_\cat{Q} 
%m\preceq_\cat{Q} p$)
%\item $\one\preceq_\cat{Q} m\mot_\cat{M} m$
%\end{itemize}

\section{$\cat{Q}$-categories}\label{sec:Q-cat}
In this section, we introduce $\cat{Q}$-categories for a quantale $\cat{Q}$. They are instances of the well-studied notion of enriched category \cite{Kelly:enriched}.

\begin{definition}
Let $\cat{Q}=(Q,\preceq_\cat{Q},\one_\cat{Q},\tms_\cat{Q})$ be a quantale. 
A \defemph{$\cat{Q}$-category} $\cat{C}$ consists of:
\begin{description}
\item[(CD1)] a set $\ob{\cat{C}}$ of \defemph{objects};
\item[(CD2)] for each $c,c'\in\ob{\cat{C}}$, an element $\cat{C}(c,c')\in {Q}$
\end{description}
satisfying the following axioms:
\begin{description}
\item[(CA1)] for each $c\in\ob{\cat{C}}$, $\one_\cat{Q} \preceq_\cat{Q} \cat{C}(c,c)$;
\item[(CA2)] for each $c,c',c''\in\ob{\cat{C}}$, $\cat{C}(c',c'')\tms_\cat{Q} \cat{C}(c,c')\preceq_\cat{Q} 
\cat{C}(c,c'')$.
\end{description}
We also write $c\in\cat{C}$ for $c\in\ob{\cat{C}}$.
\end{definition}

Note that a $\cat{Q}$-category $\cat{C}$ can be identified with a $\cat{Q}$-matrix $\cat{C}\colon \ob{\cat{C}}\pto \ob{\cat{C}}$
with $\cat{C}_{c',c}=\cat{C}(c,c')$. 
In fact, a $\cat{Q}$-category may be defined more succinctly as a set $A$ (corresponding to (CD1)) together with a $\cat{Q}$-matrix $X\colon A\pto A$ (CD2) on it such that $I_A\preceq_{A,A} X$ (CA1) and $X\tms X\preceq_{A,A}X$ (CA2) hold \cite{Betti_et_al_var_enr}.

\begin{example}
	In the case $\cat{Q}=\TV$,  we may identify the data of a $\TV$-category $\cat{C}=(\ob{\cat{C}}, (\cat{C}(c,c'))_{c,c'\in\ob{\cat{C}}})$  with a set $\ob{\cat{C}}$ equipped with a binary relation $\preceq_\cat{C}$ on it (defined as the set of all pairs $(c,c')\in\ob{\cat{C}}\times\ob{\cat{C}}$ with $\cat{C}(c,c')=\top$).
	The axioms (CA1) and (CA2) for a $\TV$-category then translate to reflexivity and transitivity of $\preceq_\cat{C}$ respectively, hence a $\TV$-category is nothing but a \emph{preordered set}.
\end{example}

\begin{example}\label{ex:nonnegRmin_cat_gms}
	In the case $\cat{Q}=\nonnegRmin$, we may regard $\nonnegRmin$-categories as \emph{generalised metric spaces} \cite{Lawvere_metric} (cf.~Example \ref{ex:d_tight_span}); 
	objects of an $\nonnegRmin$-category $\cat{C}$ are thought of as \emph{points} and the element $\cat{C}(c,c')\in[0,\infty]$ as the \emph{distance from $c$ to $c'$}.
	Notice that the axioms for $\nonnegRmin$-category indeed translate to some of the axioms for metric spaces:
	\begin{description}
		\item[(CA1)] for each $c\in\ob{\cat{C}}$, $0\geq \cat{C}(c,c)$ (that is, $\cat{C}(c,c)=0$); and
		\item[(CA2)] for each $c,c',c''\in\ob{\cat{C}}$, $\cat{C}(c',c'')+ \cat{C}(c,c')\geq \cat{C}(c,c'')$ (the triangle inequality).
	\end{description}
Every metric space is an $\nonnegRmin$-category, but not conversely. $\nonnegRmin$-categories are more general than metric spaces in the following three aspects:
\begin{itemize}
	\item distance may attain $\infty$;
	\item distance is non-symmetric (or \emph{directed}), i.e., $\cat{C}(c,c')$ may not be equal to $\cat{C}(c',c)$; and
	\item $\cat{C}(c,c')=\cat{C}(c',c)=0$ does not necessarily imply $c=c'$.\qedhere
\end{itemize}
\end{example}

\begin{example}
	Similarly, $\nonnegRminmax$-categories may be regarded as \emph{generalised ultrametric spaces}; note that axiom (CA2) now reads:
	\begin{description}
		\item[(CA2)] for each $c,c',c''\in\ob{\cat{C}}$, $\max\{\,\cat{C} (c',c''), \cat{C}(c,c') \,\}\geq \cat{C} (c,c'')$. \qedhere
	\end{description}
\end{example}

\medskip

As we have already mentioned, every left complete semimodule over $\cat{Q}$ gives rise to a $\cat{Q}$-category. The precise construction is as follows.

\begin{definition}\label{def:comp_semimod_to_Qcat}
Let $\cat{M}=(M,\preceq_\cat{M},\act_{\cat{M}})$ be a left complete semimodule over $\cat{Q}$. The following data defines a $\cat{Q}$-category, again called $\cat{M}$:
	\begin{description}
		\item[(CD1)] the set of objects is $\ob{\cat{M}}=M$; and
		\item[(CD2)] for each $m,m'\in M$, the element $\cat{M}(m,m')\in Q$ is defined as $m'\mot_\cat{M} m$ (note the order of the arguments $m$ and $m'$).
	\end{description}
	The axioms can be checked easily by means of the adjointness relations (\ref{eqn:residual_adj_semimod}).
\end{definition}

The structure of a left complete semimodule $\cat{M}$, namely the order relation $\preceq_\cat{M}$ and the action $\act_\cat{M}$, can be recovered from the induced $\cat{Q}$-category $\cat{M}=(M,(\cat{M}(m,m'))_{m,m'\in M})$.
To see this, first note that any $\cat{Q}$-category $\cat{C}=(\ob{\cat{C}},(\cat{C}(c,c'))_{c,c'\in \ob{\cat{C}}})$ determines a preorder relation $\preceq_{\cat{C}}$ on $\ob{\cat{C}}$ by 
\begin{equation}\label{eqn:preorder_on_Qcat}
	c\preceq_{\cat{C}} c' \ \text{ if and only if }\ \one_\cat{Q}\preceq_\cat{Q} \cat{C}(c,c').
\end{equation}
Two objects $c,c'\in\cat{C}$ are said to be \defemph{isomorphic} if both $c\preceq_\cat{C} c'$ and $c'\preceq_\cat{C} c$ hold.
Isomorphic objects behave exactly in the same manner, namely if $c,c'\in\cat{C}$ are isomorphic then for every object $d\in\cat{C}$, we have $\cat{C}(c,d)=\cat{C}(c',d)$ and $\cat{C}(d,c)=\cat{C}(d,c')$.
We call a $\cat{Q}$-category $\cat{C}$ \defemph{skeletal} if isomorphic objects in $\cat{C}$ are equal, i.e., if the induced preorder relation $\preceq_\cat{C}$ on $\ob{\cat{C}}$ is antisymmetric and is a partial order relation.
Now, if a $\cat{Q}$-category $\cat{M}$ is derived from a left complete semimodule $\cat{M}$ by the construction of Definition~\ref{def:comp_semimod_to_Qcat}, then we actually recover the original order relation on the complete semimodule by (\ref{eqn:preorder_on_Qcat}) (see (\ref{eqn:residual_adj_semimod})).
The action $\ast_\cat{M}$ is then uniquely determined from the hom-functor $\cat{M}(-,-)$ via the adjunction $(-)\ast_\cat{M}m\dashv \cat{M}(m,-)$.

In fact, it is also possible to define the counterpart of $\ast_\cat{M}$ for general $\cat{Q}$-categories.
Given a $\cat{Q}$-category $\cat{C}$, an object $c\in\cat{C}$ and an element $x\in Q$, an object $c'\in\cat{C}$ is said to be a \defemph{copower of $c$ by $x$} if for any $d\in\cat{C}$, the equation
\[
\cat{C}(c',d)=\cat{C}(c,d)\mot x
\]
holds \cite[Section~3.7]{Kelly:enriched}.
Copowers of $c$ by $x$ may or may not exist in $\cat{C}$, but when they exist they are unique up to isomorphism: if $c'$ is a copower of $c$ by $x$, then an object $c''\in\cat{C}$ is also a copower of $x$ and $c$ if and only if $c'$ and $c''$ are isomorphic. 
In particular, in a sketetal $\cat{Q}$-category copowers are unique.

There is also a dual notion of \defemph{power of $c\in \cat{C}$ by $x\in Q$}, which is defined as an object $c'\in\cat{C}$ such that for any $d\in\cat{C}$, the equation
\[
\cat{C}(d,c')=x\mto \cat{C}(d,c)
\]
holds.

\begin{definition}[{\cite[Section 2]{Stubbe_tensor_cotensor}}]
	A $\cat{Q}$-category $\cat{C}$ is said to be:
	\begin{itemize}
		\item \defemph{copowered} if for any $x\in Q$ and $c\in\cat{C}$, a copower $c$ by $x$ exists in $\cat{C}$; 
		\item \defemph{powered} if for any $x\in Q$ and $c\in\cat{C}$, a power of $c$ by $x$  exists in $\cat{C}$; 
		\item \defemph{order-complete} if the preordered set $(\ob{\cat{C}},\preceq_\cat{C})$ is complete (i.e., if it admits suprema of arbitrary subsets);
		\item \defemph{complete} if it is powered, copowered and order-complete.\qedhere
	\end{itemize}
\end{definition}

In fact, the construction in Definition \ref{def:comp_semimod_to_Qcat} provides a one-to-one correspondence between \emph{left complete semimodules over $\cat{Q}$} on the one hand, and \emph{skeletal and complete $\cat{Q}$-categories} on the other \cite[Section 4]{Stubbe_tensor_cotensor}. See also \cite[Section~5]{Willerton_tight_span}.

\medskip

Recall that a $\cat{Q}$-category $\cat{C}$ may be identified with a $\cat{Q}$-matrix $\cat{C}\colon\ob{\cat{C}}\pto\ob{\cat{C}}$. We can then consider its Isbell hull $\Isb(\cat{C})$; by Proposition~\ref{prop:Isbell_hull_as_hull}, it is a complete semimodule over $\cat{Q}$, hence by and Definition~\ref{def:comp_semimod_to_Qcat} we may view $\Isb(\cat{C})$ as a (skeletal and complete) $\cat{Q}$-category.
A more explicit description of $\Isb(\cat{C})$ as a $\cat{Q}$-category is the following.

\begin{definition}
	Let $\cat{C}$ be a $\cat{Q}$-category. 
	The $\cat{Q}$-category $\Isb(\cat{C})$ is defined as follows:
	\begin{description}
		\item[(CA1)] its set of objects is $\Isb(\cat{C})$ as in Definition~\ref{def:Isbell_hull};
		\item[(CA2)] given $(X,Y)$ and $(X',Y')$ in $\Isb(\cat{C})$, the element $\Isb(\cat{C})((X,Y),(X',Y'))\in Q$ is defined as 
		\[
		\bigwedge_{c\in\cat{C}}X'_{\ast,c}\mto X_{\ast,c},
		\]
		or equivalently as 
		\[
		\bigwedge_{c\in\cat{C}} Y'_{c,\ast}\mot Y_{c,\ast}.
		\]
	\end{description}
	
	The $\cat{Q}$-category $\Isb(\cat{C})$ is called the \defemph{MacNeille completion of $\cat{C}$} (\cite[Definition 7.2]{Garner_topological}, \cite[Definition 5.5.2]{Shen_thesis}).
\end{definition}

In fact, it is known that for any skeletal and complete $\cat{C}$, $\Isb(\cat{C})$ is canonically isomorphic to $\cat{C}$ (see \cite[Proposition 5.5.5]{Shen_thesis} or \cite[Proposition 7.6]{Garner_topological}).
It follows that any skeletal and complete $\cat{Q}$-category, or equivalently any left complete semimodule over $\cat{Q}$, can be realised the Isbell hull $\Isb(Z)$ of a certain $\cat{Q}$-matrix $Z$.

%-------------------
\bibliographystyle{alpha} %
\bibliography{myref} %
\end{document}